\documentclass[12pt]{article}

\usepackage{amssymb}
\usepackage{amsthm}
\usepackage{amsmath} %txfonts mathtools
\usepackage{stmaryrd}
\usepackage{mathrsfs,bm,hyperref}
\usepackage{soul,xcolor}
\usepackage{mathrsfs}

\usepackage{tikz}
\usepackage{tikz-3dplot}
\usetikzlibrary{shapes.geometric}
\usepackage{pgfplots}
\pgfplotsset{compat=1.18}
\usetikzlibrary{patterns}
\usetikzlibrary{perspective}
\usetikzlibrary{arrows}
\usetikzlibrary{3d,calc}
\usetikzlibrary{angles, calc, intersections, quotes, arrows.meta}
\usepackage{accents}
\usepackage{rotating}
\usepackage{soul}
\usepackage{graphicx}

\usepackage{mathtools}
\usepackage{enumitem}
\usepackage{isomath}

\usepackage{cite}
\usepackage[us,12hr]{datetime}

\usepackage{dsfont}
\usepackage[scr=boondox]{mathalpha}
\usepackage{nicematrix}
\newcommand{\bpsi}{\bm \psi}

\newcommand{\de}{\partial}
\renewcommand{\div}{\mathrm {div\,}}

\newcommand{\ep}{\varepsilon}
\newcommand{\ept}{\hat{\ep}}
\renewcommand{\leq}{\leqslant}

\newcommand{\mut}{\hat{\mu}}

\newcommand{\n}{\bm n}

\newcommand{\om}{\omega}
\newcommand{\pperp}{B}
\newcommand{\rf}[1]{(\ref{#1})}
\renewcommand{\rho}{\varrho}

\newcommand{\curl}{\mathrm{curl\,}}

\renewcommand{\u}{\bm u}
\newcommand{\ut}{\tilde{\u}}

\renewcommand{\v}{\bm v}

\newcommand{\w}{\bm w}

\newcommand{\x}{{\bm x}}
\newcommand{\xt}{\hat{\x}}

\newcommand{\z}{{\bm 0}}

\DeclareMathAlphabet{\msfsl}{OT1}{cmss}{m}{sl}

\newcommand{\B}{B}
\newcommand{\Bb}{\bm B}
\newcommand{\C}{{\mathbb C}}

\newcommand{\Db}{\bm D}
\newcommand{\Eb}{\bm E}

\newcommand{\Fb}{\bm F}

\newcommand{\Hb}{\bm H}
\newcommand{\K}{\bm K}
\renewcommand{\L}{\bm {\mathscr L}}
\newcommand{\M}{M}

\newcommand{\Om}{\Omega}

\newcommand{\Q}{\bm {\mathscr Q}}
\newcommand{\R}{{\mathbb R}}
\newcommand{\Res}{\bm {\mathscr R}}
\newcommand{\T}{{\mbox{\tiny \textsf{T}}}}

\newcommand{\V}{\bm V}

\newcommand{\W}{\bm W}

\newtheorem{theorem}{Theorem}[section]

\newtheorem{lemma}{Lemma}[section]

\newtheorem{remark}{Remark}

\newcommand{\I}{\ensuremath{\mathrm{i}}}
\newcommand{\E}{\ensuremath{\mathrm{e}}}

\DeclareMathAlphabet{\mathsfit}{T1}{\sfdefault}{m}{sl}
\SetMathAlphabet{\mathsfit}{bold}{T1}{\sfdefault}{bx}{sl}

\begin{document}
\setstcolor{red}

\title{Embedding transmission problems for Maxwell's \\ equations into elliptic theory}
\author{Yuri A. Godin\footnote{email: ygodin@charlotte.edu} ~and Boris Vainberg\footnote{email: brvainbe@charlotte.edu} \vspace{5mm} \\
Department of Mathematics and Statistics, \\ University of North Carolina at Charlotte, \\
Charlotte, NC, USA}

\date{\today}

\maketitle

\begin{abstract}
We embed general boundary value problems for the time-harmonic Maxwell equations into the elliptic boundary value theory. This is achieved by introducing two new scalar functions to the electromagnetic field and imposing additional boundary conditions, after which the problem becomes elliptic. The results are applied to general problems for Maxwell's equations in bounded and unbounded domains, as well as to the transmission problem with inhomogeneities on the right-hand side of the equations and at all boundaries. Relations between the inhomogeneities of the elliptic problem are established that provide a one-to-one correspondence between the solutions of Maxwell's problem and the solutions of the elliptic boundary value problem. 
As an application, the existence of an orthogonal basis from the eigenfunctions of the general Maxwell's problem is justified.
\end{abstract}

% \bigskip
% {\bf Keywords:} 
% Maxwell equations;
% Transmission problem;
% Shapiro-Lopatinsky condition; \\

% MATHEMATICS SUBJECT CLASSIFICATION: \\

% \noin
% 35Q61  	Maxwell equations \\
% 35B65  	Smoothness and regularity of solutions to PDEs\\
% 35Q60  	PDEs in connection with optics and electromagnetic theory\\
% 35J57  	Boundary value problems for second-order elliptic systems\\
% 35B65  	Smoothness and regularity of solutions to PDEs\\

\section{Introduction}
\setcounter{equation}{0}
\label{I}

It is well known that the system of time-harmonic Maxwell's equations is not elliptic, and, therefore, a boundary value problem (BVP) for the Maxwell system cannot be elliptic. Many (but not all) general statements of elliptic theory have been independently proved for Maxwell’s equations; see, for example, \cite{Muller:1969,  Weber:1981, Alberti:2014, Alberti:2018, ColtonKress:2019}. Reducing boundary value problems for Maxwell's equations to an elliptic problem allows one to get these results immediately by appealing to elliptic theory and obtain new facts like smoothness of the solutions, local a priori estimates, reduction to integral equations, asymptotics of the spectrum, and more. 
It is therefore not surprising that there is interest in incorporating Maxwell's equations into elliptic theory \cite{Picard:1986, Costabel:1991, Weck:1974, Birman:1987, Birman:1989, Monk:2003, ColtonKress:2019}.

We will consider general BVPs for Maxwell's system. Its reduction to an elliptic problem is trivial in the simple case of the Maxwell equations in the whole space
\begin{align}
\label{MR}
 \curl \Eb &= \I \om \mu \Hb,~~ \quad
  \curl \Hb = - \I \om \ep \Eb,~~\quad \x\in\mathbb R^3,    
\end{align}
with constant dielectric permittivity  $\ep$ and magnetic permeability $\mu$ of the medium. By applying the operator $\curl$ to the equations above, one can reduce them to separate Helmholtz equations for $\Eb $ and $\Hb$.
\[
\Delta \Eb+\om^2\mu\ep\Eb=0, ~~\quad \Delta \Hb+\om^2\ep\mu\Hb=0,~~\quad \x\in\mathbb R^3.
\]
This approach can be generalized to the case when $\ep$ and $\mu$ are smooth functions (or matrix functions) of $\x$. The result is an elliptic system for $(\Eb,~\Hb)$ instead of a separate equation for each component. 
However, such a reduction to an elliptic system is possible only in the entire space $\mathbb{R}^3$, and even then, there is a drawback: differentiating the original problem requires additional smoothness of the data.

The main difficulties arise when Maxwell's equations are considered in a domain $\Om\in\mathbb R^3$ (bounded or unbounded), since it is unclear what additional boundary conditions need to be added to the problem for the Helmholtz equations to ensure both the ellipticity of the BVP and the connection with the original BVP for Maxwell's equations. 
We recall that the ellipticity of a BVP requires not only the ellipticity of the equations but also the ellipticity of the boundary conditions, i.e., fulfillment of the Shapiro-Lopatinsky condition \cite{Lopatinsky:1953, Shapiro:1953, ADN:59, Lions:1972, Egorov:1992} (also called the Agmon-Douglis-Nirenberg condition or the covering condition for the boundary operators). A detailed description of this condition will be provided in the proof of Theorem \ref{t1}.

The situation worsens significantly for the transmission problem,
when $\Om $ contains a subdomain $\Om_-$ and the matrix functions $\ep$ and $\mu$ have discontinuities at the interface $\Gamma = \de\Om_-$. Then, specific interface conditions hold on $\Gamma$ for Maxwell's problem, and additional boundary conditions on $\Gamma$ should be introduced for the solutions of the Helmholtz equation. 
It is not clear what additional conditions ensure the ellipticity of the problem for the Helmholtz equation and the connection between the elliptic problem and the transmission problem for Maxwell's equations.

In \cite{Picard:1986, Birman:1987, Birman:1989}, a different approach was proposed to ellipticize Maxwell's equations.
Instead of reducing to Helmholtz equations, the authors add
 two unknown scalar functions $\alpha, \beta$ to the electromagnetic field $(\Eb, \Hb)$ and consider an extended system of eight first-order equations with eight unknowns $(\Eb, \alpha, \Hb, \beta)$
 in $\Om$ with variable $\ep$ and $\mu$. The obtained system is elliptic in the interior domain in the Petrovsky (Agmon-Douglis-Nirenberg) sense. The authors do not consider transmission problems for Maxwell's equations and also restrict themselves to a perfectly conducting boundary. Under these constraints, the boundary conditions on $\de \Om$ for Maxwell's problem are homogeneous, and this allows one to define the domain of the operator of the constructed elliptic system simply as the closure in the Sobolev space $H^1$ of smooth functions in $\R^3$, vanishing outside $\Om$.

For the same elliptic system of eight equations involving $(\Eb, \alpha, \Hb, \beta)$, we construct a class of nonhomogeneous boundary conditions with inhomogeneities on the right-hand side of the equations and the boundaries $\Gamma = \de \Om_-, ~\de \Om$, that make the entire BVP elliptic, and this problem has the following property.
Let $\Om$ be a bounded domain with a subdomain $\Om_{-}$ located strictly inside $\Om$,  and
% $\Om$ of the form 
% $\Om=\bar{\Om}_-\bigcup\Om_+$, 
$\Om_+ = \Om \setminus \bar{\Om}_-$, see Figure \ref{domain}.
Consider BVPs for Maxwell's equations with solutions in the Sobolev space $H^1 (\Om \setminus \Gamma)$, and with nonhomogeneous boundary conditions at $\Gamma$ and $\de \Om$ belonging to the Sobolev space $H^{1/2}$. 
For {\it any given Maxwell problem}, we explicitly describe the inhomogeneities (right-hand sides of the equations and boundary conditions) of the extended problem, such that a simple relation $(\Eb, \Hb) \leftrightarrow (\Eb, 0, \Hb, 0)$ holds between the solutions of the Maxwell problem and the solutions of the elliptic BVP.

The arguments and results of the paper obtained below for the transmission problem for the time-harmonic electromagnetic field in a bounded domain $\Om$
can be applied without any changes to problems in 
the external region $\R^3 \setminus \Om$ or the entire space when the matrices $\ep$ and $\mu$ approach fast enough to constant matrices at infinity and the appropriate radiation conditions are imposed. 

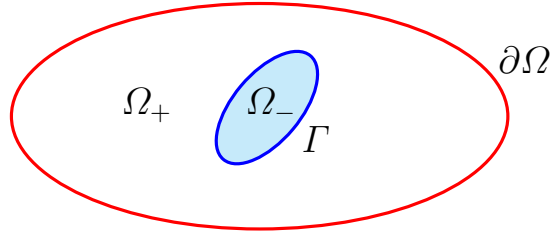
\begin{figure}[ht]
\begin{center}
\begin{tikzpicture}[scale=1.5,>=Stealth]

\draw[red, very thick] (0,0) ellipse (2.2cm and 1cm);
\draw[red, very thick, blue,text = black, fill = cyan!20,rotate=-40] (0,0.1) ellipse (0.3cm and 0.6cm);

\node at (-1,0.1) {\mbox{\large $\Om_{+}$}};
\node at (2.35,0.5) {\large $\de \Om$};
\node at (0.1,0.1) {\large $\Om_{-}$};
% \node at (0.35,0.4) {\large $\ep_{-}$};
% \node at (-0.1,-0.25) {\large $\mu_{-}$};
% \node at (1.3,0.1) {\large $\ep_{+}, \mu_{+}$};
\node at (0.5,-0.2) {\large $\Gamma$};

\end{tikzpicture}
\end{center}
\vspace{-5mm}
\caption{
Geometry of the transmission problem. The domain $\Om$ contains an inclusion $\Om_{-}$, bounded by $\Gamma = \de \Om_{-}$, and $\Om_{+} = \Om \setminus \bar{\Om}_-$.}
\label{domain}
\end{figure}
 
The boundaries $\Gamma = \de \Om_{-}$ and $\de\Om$ are assumed to be Lipschitz. The material parameters under consideration $\ep = \ep (\x)$ and $\mu = \mu (\x)$ can be either complex-valued functions with positive real parts or matrix functions with positive-definite real parts (the medium is anisotropic in the latter case). 

Denote $\ep=\ep_\pm, \mu=\mu_\pm$ when $\x\in\Om_\pm$.
It is assumed that
\begin{align}
\label{w1i}
 \ep_\pm,~\mu_\pm \in W^{1,\infty}(\Om_\pm).
\end{align}
The latter space consists of matrix-valued functions whose entries and first derivatives are essentially bounded in $\Om_\pm$. 

To make our approach more transparent, we consider, in the next section, a simple transmission problem for Maxwell's equations with inhomogeneities only on $\de \Om$ and reduce this problem to an elliptic one.  In section \ref{Rel}, we establish a one-to-one correspondence between the solutions of Maxwell's equations with inhomogeneities only on $\de \Om$ and solutions of the elliptic problem. At the end of the section, these results are extended to the case when the inhomogeneities are also present at the interface. The general problems (with charges and currents present in the equations) are considered in section \ref{sec4}. Section \ref{sec5} addresses the application of the embedding of Maxwell's equations into elliptic theory. We utilized the meromorphic dependence of the resolvent of the constructed elliptic problem on the parameter $\om$ to prove the discreteness of the spectrum of the Maxwell problem in a bounded domain. 
In the case where $\ep$ and $\mu$ are real positive-definite matrices, the existence of an orthogonal basis of eigenfunctions for the Maxwell problem is established.

All vectors in the paper are understood as column vectors, although we write them for convenience as row vectors.

\section{Reduction of Maxwell's equations to an elliptic problem}
\setcounter{equation}{0}
\label{R}

 We start with Maxwell's equations in the absence of charges and currents. Then the problem has the form
\begin{align}
\label{Mone}
 \curl \Eb &= \I \om \mu \Hb,~~ \quad
  \curl \Hb = - \I \om \ep \Eb,~~\quad \x\in\Om\setminus\Gamma, \quad \Gamma=\de\Om_-.
\end{align}
where $\ep = \ep(\x)$ and $\mu = \mu (\x)$ are the permittivity and permeability of the medium, respectively, satisfying the conditions imposed in the Introduction. It is convenient to rewrite equations \rf{Mone} in the matrix form
\begin{align}
\label{Mone-vec}
    \I \left(
    \begin{array}{cc}
        0 & \curl  \\
   -\curl & 0
    \end{array}
    \right) \left(
    \begin{array}{c}
        \Eb \\
        \Hb
    \end{array}
    \right) = \om \left(
    \begin{array}{cc}
        \ep & 0  \\
          0 & \mu
    \end{array}
    \right) \left(
    \begin{array}{c}
        \Eb \\
        \Hb
    \end{array}
    \right).
\end{align}

From now on, the notation $H^s$ will be used for Sobolev spaces of functions as well as for Sobolev spaces of vector functions. We will look for solutions $\Eb,\Hb$ in the space  $H^1(\Om\setminus\Gamma)$ that consists of functions defined on $\Om$ whose restrictions to $\Om_\pm $ belong to $H^1(\Om_\pm)$, respectively. Note that $\Gamma$ and $\de\Om$ are Lipschitz surfaces, and therefore, the trace operators 
\begin{align*}
    t_\pm : ~H^1(\Om\setminus\Gamma) \to H^{1/2}(\Gamma), \quad
    t_0 : ~H^1(\Om\setminus\Gamma) \to H^{1/2}(\de \Om),
\end{align*}
are bounded. Here, the subscript $\pm$ indicates whether the trace is taken from $\Om_+$ or from $\Om_-$.
 
 The choice of the space for $\ep_\pm, \mu_\pm$ is also very natural since this space is the multiplier algebra of $H^1$,
and the inclusion \rf{w1i} is both sufficient and (essentially) necessary for $\ep\Eb, \mu\Hb\in H^1(\Om\setminus\Gamma)$ with arbitrary functions  $\Eb, \Hb\in H^1(\Om\setminus\Gamma)$.
 
Let us describe the boundary conditions on $\Gamma$ and $\de\Om$. Denote by $\n$ the unit normal vector to $\Gamma$ and $\de\Om$ that is exterior to $\Om_-$ and
$\Om$, respectively. Let $\Eb_\tau=\n\times \Eb$, $\Hb_\tau=\n\times \Hb$ be the tangential components of $\Eb,\Hb$ rotated by $90^\circ$ in the tangent plane. Note that if the $x_3$-axis is directed along $\n$ and $\Eb=(E_1,E_2,E_3)$, then $\n\times \Eb=(-E_2,E_1,0)$ (and the tangential component $(E_1,E_2,0)$ without additional rotation is given by $\n\times(\Eb \times \n)$).
Denote by $\llbracket u \rrbracket$ the jump of the function $u$ across $\Gamma$, equal to the limiting value for $\x \in \Om_{-}$ subtracted from the limiting value for $\x \in \Om_{+}$. 

The transmission problem consists in finding solutions
$\Eb, \Hb$ of \rf{Mone} in the space $ H^1(\Om\setminus\Gamma)$ that satisfy the boundary conditions:
\begin{align}
\label{Ma}
 \llbracket\Eb_\tau\rrbracket=&0, \quad 
  \llbracket\Hb_\tau\rrbracket=0, \quad \x\in \Gamma, \\[2mm]
 \Eb_\tau=&\Eb_\tau^0\in H^{1/2} (\de\Om). 
 \label{Mc}
\end{align}

In most textbooks and other publications on Maxwell's equations, the boundary conditions \rf{Ma}, \rf{Mc} also include conditions on the scalar values $D_\nu=\n\cdot(\ep\Eb),~ B_\nu=\n\cdot(\mu\Hb)$ of the normal components of the electric displacement field $\Db = \ep \Eb$ and magnetic induction $\Bb = \mu \Hb$. It is not always a mistake, but one should keep in mind that equations \rf{Mone} relate the values of $ B_\nu, D_\nu$ and $\Eb_\tau, \Hb_\tau$, respectively. These relations also include the trace at the boundary of the nonhomogeneous right-hand side in  \rf{Mone} if these terms are added to \rf{Mone}. Thus, the normal components of the fields at the boundary cannot be specified independently but must be evaluated through the tangential components.

The boundary value problem \rf{Mone}-\rf{Mc} is not elliptic, and the system \rf{Mone} itself is not elliptic. 
% Many (but not all) general statements of elliptic theory have been independently proved for Maxwell's equations. 
Some particular cases in which the problem \rf{Mone}-\rf{Mc} can be replaced by an elliptic one have been discussed in the introduction. An equivalent approach is based on the representation of the electric field as a sum of vector and scalar potentials and a reduction of the original problem to an elliptic one for the potentials. These approaches do not allow one to consider the transmission problem and break the symmetry between $\Eb$ and $\Hb$, affecting the smoothness result for $\Hb$. 

Consider the following extended version of the problem \rf{Mone}-\rf{Mc}. The new system contains two additional scalar functions $\alpha, \beta$
\begin{alignat}{3}
\label{L1}
 \curl \Eb+\mu \nabla \alpha &= \I \om \mu \Hb,~~ &\quad~
  \curl \Hb+\ep \nabla\beta &= - \I \om \ep \Eb,~~&\quad~ &\x\in\Om\setminus\Gamma \\[2mm]
  \label{L2}  -\div(\ep\Eb)&=0, &\quad  -\div(\mu\Hb)& =0,~&\quad~ &\x\in\Om\setminus\Gamma.
\end{alignat}
The minus signs in \rf{L2} are added to make the operators symmetric in section \ref{sec5} devoted to the resolvents.
We look for solutions $\u = (\Eb, \alpha, \Hb , \beta)$ of \rf{L1}, \rf{L2} whose components belong to the vector or scalar space $H^1(\Om\setminus\Gamma)$ and satisfy the following boundary conditions:
\begin{alignat}{4}
\label{La}
 \llbracket\Eb_\tau\rrbracket&=\Eb_\tau^\Gamma, &\quad ~ \llbracket\n\cdot \ep \Eb\rrbracket&=D_\nu^\Gamma, &\quad ~\llbracket\beta\rrbracket &=& \beta^\Gamma , \quad~&\x\in \Gamma, \\[2mm]
 \label{Lb}
\llbracket\Hb_\tau\rrbracket&=H_\tau^\Gamma, &\quad ~ \llbracket\n\cdot \mu \Hb\rrbracket&=B_\nu^\Gamma, &\quad ~\llbracket\alpha\rrbracket&=& \alpha^\Gamma , \quad~&\x\in \Gamma, \\[2mm]
 \Eb_\tau &= \Eb_\tau^0, & \quad \n\cdot \mu \Hb&=B_\nu^0, &\quad~  \beta&=&\beta^0, \quad~ &\x\in \de\Om,
 \label{Lc}
\end{alignat}
where the right-hand sides in \rf{La}-\rf{Lc} are prescribed and belong to the (vector or scalar) space $ H^{1/2} (\Gamma)$ in \rf{La}, \rf{Lb} and $ H^{1/2} (\de\Omega)$ in \rf{Lc}. Similarly to \rf{Mone-vec}, it is expedient to rewrite \rf{L1},\rf{L2} as follows: 
\begin{align}
    \L \u := \I \left(
    \begin{array}{cc}
        \z & \L_{\ep,\mu}  \\ [2mm]
   -\L_{\mu,\ep} & \z
    \end{array}
    \right) \u = \om \B \,
     \u,
\end{align}
where
\begin{align}
\label{Lem}
     \L_{\ep,\mu} = \left(
    \begin{array}{cc}
      \curl & \ep \nabla \\[2mm]
      -\div (\mu \cdot) & 0
    \end{array}
    \right), \quad
    \B = \begin{pNiceArray}{cccc}
  \ep &  & \Block{2-2}< >{\z} \\[-1mm]
   & 0 \\
  \Block{2-2}< >{\z} && \mu &  \\[-1mm]
  &&  & 0
\end{pNiceArray},
\end{align}
and $\L_{\mu,\ep}$ is obtained from $\L_{\ep,\mu}$ by interchanging $\ep$ and $\mu$.
\begin{remark}\label{R1}Note that the extended problem has two additional unknown scalar functions, $\alpha$ and $\beta$,  two additional equations \rf{L2}, and the number of boundary conditions increases substantially. For example, the transmission problem for the Maxwell equation has four scalar relations for fields on $\Gamma$ while the extended problem has eight boundary conditions on $\Gamma$.
\end{remark}
To study the relation between problem \rf{L1}-\rf{Lc} and the transmission problem for the Maxwell equations, we need to express the normal component of $\curl\Fb$ for an arbitrary sufficiently smooth vector field $\Fb$ in a neighborhood of $\Gamma$ or $\de \Om$ through the tangential component $\Fb_\tau$ of the field. Let us fix a point $\x_0$ on the boundary and choose the system of coordinates for which the directions of the axis $x_3$ and $\n$ coincide. If $\Fb=(F_1,F_2,F_3)$ in the new coordinate system, then the scalar value $(\curl\Fb)_\nu=\n\cdot \curl \bm F$ of the last (normal) component of $\curl\Fb$ is $(F_2)'_x-(F_1)'_y$. We denote the last expression by $-\div_\tau\Fb_\tau$ since $\Fb_\tau=(-F_2,F_1,0)$ (see the second paragraph of section \ref{R}). This expression contains the tangential derivatives of the tangential components of the field $\Fb$ and does not depend on the rotation of the coordinate system. Thus, the following statement is valid.
\begin{lemma} \label{L00} 
For an arbitrary sufficiently smooth vector field $\Fb$, the scalar value of the normal component of $\curl\Fb$ can be expressed through the tangential derivatives of the tangential components of the field $\Fb$, and the following relation is valid:
\begin{align}
\label{l3.1}
\n\cdot\curl\Fb:=(\curl\Fb)_\nu=-\div_\tau\Fb_\tau, \quad \x\in \Gamma ~~ or ~~\de\Om.
\end{align}
If $\Fb \in H^1(\Om\setminus\Gamma)$, then the above relation is valid in the space $H^{-1/2} (\Gamma)$ and in $H^{-1/2} (\de\Om)$.
\end{lemma}
\begin{remark}\label{re2}
The derivatives of $\Fb \in H^1(\Om\setminus\Gamma) $ may not have the trace on $\Gamma$ or $\de\Om$, but the normal component of $\curl \Fb$ has the trace. It can be defined using \rf{l3.1} by taking $\Fb_\tau \in H^{1/2}$ at the boundary and applying $-\div_\tau$. Alternatively, one can use 
an approximation of $\Fb$ by smooth vector fields followed by passing to the limit in \rf{l3.1}. 
\end{remark}
\begin{theorem} \label{t1}
Let $\ep = \ep (\x)$ and $\mu = \mu (\x)$ be either complex-valued functions with a positive real part or real-valued positive-definite matrix functions, and let \rf{w1i} hold. Then the boundary value problem \rf{L1}-\rf{Lc} is elliptic.
\end{theorem}
\begin{proof}
We need to show the ellipticity of the system of equations \rf{L1}, \rf{L2}, and the ellipticity of the boundary conditions. For the system, one needs to fix an arbitrary point $\x\in \bar \Om$, omit the lower-order terms in \rf{L1}, \rf{L2}, and show that the characteristic matrix of the obtained system is invertible when the dual variable is not at the origin. After dropping the lower-order terms, the right-hand sides in equations \rf{L1} become equal to zero, 
and the factors $\ep,\mu$ in \rf{L2} are omitted. 
Hence, system \rf{L1}, \rf{L2} is divided into two similar systems for $(\Eb,\alpha)$ and  for $(\Hb,\beta)$ with matrix operators $\left(\begin{array}{ccc}
 \curl & \eta \nabla \\[1mm]
-\div& 0
 \end{array}\right)$, where $\eta$ equals $\mu$ or $\ep$, respectively. We leave the proof of the invertibility of the corresponding characteristic matrix to the reader (the determinant of the latter matrix is $({\bm\sigma} \cdot \eta \,\bm\sigma)|\bm\sigma|^2$, where $\bm\sigma\in \mathbb R^3$ is the dual variable, $\eta = \mu$ or $\ep$.) 
 
 To show the ellipticity of the boundary conditions, one needs  \cite{Lopatinsky:1953, Shapiro:1953, ADN:59, Lions:1972, Egorov:1992} to fix an arbitrary point $\x_0$ on the boundary $\de\Om$ or $\Gamma$, introduce local coordinates in which the boundary near $\x_0$ is flat, put $\x=\x_0$ in the equations and boundary conditions, and drop the lower-order terms. The ellipticity (the Shapiro-Lopatinsky) condition holds if the obtained problem has only the 
 trivial solution in the class of functions of the form $\E^{\I \x'  \cdot {\bm\sigma}'} \v$, 
 where the vector $\x' = (x_1',x_2')$ consists of tangential variables, 
 $\bm\sigma'$ is the corresponding dual vector, the vector function  $\v$ depends only on the normal variable and
 decays exponentially at infinity.

If $\x_0 \in \Gamma$, then one needs to consider the problem in a neighborhood of $\x_0$ as a system of two problems posted on each side of $\Gamma$. In the local coordinates, these problems are reduced to the corresponding problems in two half-spaces. These two problems in the half-spaces with a common boundary $\Gamma'$ that are related by the boundary conditions must have only the trivial solutions in the class of functions decaying exponentially when the distance from $\Gamma'$ goes to infinity.

Let us prove the ellipticity of problem \rf{L1}-\rf{Lc} at a point $\x_0\in\Gamma$; the ellipticity at $\x_0\in\de\Om$ can be shown similarly. We move the origin to the point $\x_0$, rotate the coordinates $\x - \x_0 \to \xt = \M (\x - \x_0)$ to direct the $\hat{x}_3$-axis along the normal vector $\n$ at the point $\x_0$, and make the shift $\hat{x}_3-h(\hat{x}_1,\hat{x}_2)\to \hat{x}_3$ in such a way that $\Gamma$ near the origin coincides with the $2$-dimensional space $\Gamma'$ given by the equation $\hat{x}_3=0$.  Then we put $\xt=0$ in the coefficients of the equations and the boundary conditions and drop the lower-order terms. We come up with the following system of equations in $\mathbb R^3\setminus\Gamma'$. We use the same notation for the vectors $\Eb,  \Hb, \n$ in the new system of coordinates, use the notation $\x$ for $\xt$ and $\ept = \M \ep \M^{\T}$, $\mut = \M \mu \M^{\T}$.  Then $\Eb, \alpha, \Hb, \beta$ satisfy the equations 
\begin{alignat}{3}
\label{N1}
 \curl \Eb+ \mut (\x_0) \nabla \alpha &= \z,  &\quad \curl \Hb+\ept (\x_0) \nabla\beta &= 0, ~&\quad~ &\x\in\mathbb R^3\setminus\Gamma',\\[2mm]
 \label{N2}  -\div(\Eb)&=0, &-\div(\Hb)&=0,~&\quad~ &\x\in\mathbb R^3\setminus\Gamma',
\end{alignat}
and the boundary conditions
\begin{alignat}{4}
\label{Na}
\llbracket\Eb_\tau\rrbracket&=0,  &\quad ~\llbracket\n\cdot  \ept (\x_0) \Eb\rrbracket&=0, &\quad &\llbracket\alpha\rrbracket&= 0 , \quad~&\x\in \Gamma', \\[2mm]
 \label{Nb}
\llbracket\Hb_\tau\rrbracket&=0, &\quad ~ \llbracket\n\cdot\mut (\x_0)\Hb\rrbracket&=0, &\quad &\llbracket\beta\rrbracket&= 0 , \quad~&\x\in \Gamma',
\end{alignat}
Let us mention that equations \rf{N2} do not contain $\ep$ and $\mu$ because the lower-order terms are omitted from the equations. The problem \rf{N1}-\rf{Nb} is split naturally into two completely analogous problems for $(\Eb,\alpha)$ and  for $(\Hb,\beta)$. The first one has the form
\begin{align}\label{A1}
 &\curl \Eb+\mut (\x_0) \nabla \alpha = \z, ~~~\div(\Eb)=0,  \quad~ \x\in\mathbb R^3\setminus\Gamma',\\[2mm]\label{A2}
 &\llbracket\Eb_\tau\rrbracket=0,  \quad ~\llbracket\n\cdot  \ept (\x_0)\Eb\rrbracket=0, \quad \llbracket\alpha\rrbracket= 0 , \quad~\x\in \Gamma'.
\end{align}
We apply the operator $\div$ to the first equation in \rf{A1} and obtain
\begin{align}
\label{A3}
    \div \left(\mut (\x_0) \nabla \alpha \right) = 0, \quad~ \x\in\mathbb R^3\setminus\Gamma'.
\end{align}
It follows from the Lemma \rf{L00} and the first equations in \rf{A1},\rf{A2} that the normal component of $\mut (\x_0) \nabla \alpha$ is well-defined and has no discontinuity on $\Gamma'$. Together with the last condition in \rf{A2} this leads to the validity of \rf{A3} in the whole space $\R^3$. This elliptic equation has only a trivial solution in the class of bounded functions decaying when $|x_3| \to \infty$. Thus, $\alpha = 0$.

From the first two equations in \rf{A2} it follows that $\Eb$ does not have a discontinuity on $\Gamma'$ and therefore \rf{A1} with $\alpha=0$ is valid in the whole space $\R^3 $. Hence, $\Eb$ is harmonic in $\R^3 $. Thus, $\Eb = \z$ since it is bounded and decays as $|x_3| \to \infty$.
\end{proof}

\section{The relationship between Maxwell's equations and the elliptic problem}
\setcounter{equation}{0}
\label{Rel}

\begin{theorem} \label{t0} 
\leavevmode
\begin{itemize}[label={(\roman*)}]
\item[(i)]    
If  $(\Eb, \Hb)$ is a solution of \rf{Mone}-\rf{Mc} with $\Eb, \Hb \in  H^1(\Om\setminus\Gamma) $, then $(\Eb, \alpha, \Hb , \beta)$  with $\alpha=const.,~\beta=0$ is a solution of the elliptic problem \rf{L1}-\rf{Lc} with zero on the right-hand sides of \rf{La},\rf{Lb} and 
\begin{align}
\label{thm31}
    \I \om B_\nu^0=-\div_\tau\Eb_\tau^0, \quad \beta^0=0.
\end{align}

\item[(ii)]
Conversely, if $(\Eb, \alpha, \Hb, \beta)$ is a solution of \rf{L1}-\rf{Lc} with all the components belonging to the vector or scalar space $H^1(\Om\setminus\Gamma) $, with zero on the right-hand sides of \rf{La},\rf{Lb}, and \rf{thm31} holds, then  $(\Eb,\Hb)$ is a solution of the Maxwell problem \rf{Mone}-\rf{Mc}.
\end{itemize}
\end{theorem}
\begin{proof} The first statement is obvious (the relation for $B_\nu^0$ follows from \rf{Mone} and Lemma \ref{L00}). Let us prove the second statement. 

Applying $\div$ to \rf{L1}, we obtain
\[
\div (\mu \nabla \alpha) = \div(\I \om \mu \Hb),~~ \quad~
  \div (\ep \nabla \beta) = - \div(\I \om \ep \Eb),~~\quad~ \x\in\Om\setminus\Gamma.
\]
This and \rf{L2} imply
\begin{align}
\label{Lap}
\div (\mu \nabla \alpha) = 
  \div (\ep \nabla \beta) = 0,~~\quad~ \x\in\Om\setminus\Gamma.
\end{align}

Note that $\llbracket\de\alpha/\de \hat{\n}\rrbracket = \llbracket(\mu \nabla\alpha)_\nu\rrbracket$,
$\x\in\Gamma$, where $\hat{\n}$ is conormal to $\Gamma$ for the operator $\div (\mu \nabla \cdot)$, is well-defined in $H^{-1/2} (\Gamma)$, see \cite{McLean:2000}.
Furthermore, from Lemma \ref{L00}, the homogeneous first relation in \rf{La} and the homogeneous second relation in  \rf{Lb} we get $\llbracket(\curl\Eb)_\nu\rrbracket=\llbracket(\I\om\mu\Hb)_\nu\rrbracket$ (both terms are zero), and therefore the first equation in \rf{L1} implies $\llbracket\de\alpha/\de \hat{\n}\rrbracket=0$.
 The same arguments using $\de\Om$ instead of $\Gamma$ lead to $\de\alpha/\de \hat{\n}=0, ~\x\in \de\Om$. 
If, in the reasoning presented above, we interchange $\Eb$ and $\Hb$, we obtain $\llbracket\de\beta/\de\hat{\n}\rrbracket=0$, where $\hat{\n}$ is the conormal for the operator $\div (\ep \nabla \cdot)$. Now, the last relations in \rf{La}-\rf{Lc}  with zero on their right-hand sides allow us to apply Green's first identity to the solutions of equations \rf{Lap}; as a result, we obtain that $\alpha$ is a constant and $\beta=0$. Consequently, the pair $(\Eb,\Hb)$ satisfies the equations \rf{Mone}-\rf{Mc}.
 \end{proof}

The proof of Theorem \ref{t0} can be applied without any changes to verify the following statement.

\begin{theorem} \label{T2} 
\leavevmode
\begin{itemize}[label={(\roman*)}]
\item[(i)]    
If  $(\Eb, \Hb)$ is a solution of \rf{Mone}-\rf{Mc} with $\Eb, \Hb \in  H^1(\Om\setminus\Gamma) $ and nonhomogeneous right-hand sides  $\Eb_\tau^\Gamma,~ \Hb_\tau^\Gamma\in H^{1/2}(\Gamma)$ in  \rf{Ma}, then $(\Eb, \alpha, \Hb, \beta)$  with $\alpha=const.,~\beta=0$ is a solution of the elliptic problem \rf{L1}-\rf{Lc} with the following relations between the boundary data:
\begin{alignat}{2}
\label{hom}
\I \om B_\nu^\Gamma & =-\div_\tau\Eb_\tau^\Gamma,~&&\quad \I \om D_\nu^\Gamma=\div_\tau\Hb_\tau^\Gamma, \\[2mm]
\I \om B_\nu^0 &=-\div_\tau\Eb_\tau^0,~&&\quad \beta^\Gamma =\alpha^\Gamma=\beta^0=0.
\label{home}
\end{alignat}
\item[(ii)]
Conversely, if $(\Eb, \alpha,\Hb, \beta)$ is a solution of \rf{L1}-\rf{Lc} 
with all components in the vector or scalar space $H^1(\Om\setminus\Gamma) $ and relations \rf{hom},\rf{home} are satisfied, then  $(\Eb,\Hb)$ is a solution of the Maxwell problem \rf{Mone}-\rf{Mc} with nonhomogeneous right-hand sides $\Eb_\tau^\Gamma,~ \Hb_\tau^\Gamma$ in  \rf{Ma}.
\end{itemize}
\end{theorem}

 \section{The general nonhomogeneous transmission problem}
 \label{sec4}
 \setcounter{equation}{0}
 
 Consider now the general nonhomogeneous Maxwell's equations
 \begin{align}
\label{M1a}
 \curl \Eb &= \I \om \mu \Hb+\K,~~ \quad
  \curl \Hb = - \I \om \ep \Eb+\bm J,~~\quad \x\in\Om\setminus\Gamma,
\end{align}
with nonhomogeneous boundary conditions on $\Gamma$ and $\de\Om$:
\begin{align}
\label{Maa}
\llbracket\Eb_\tau\rrbracket=&\Eb_\tau^\Gamma, \quad \llbracket\Hb_\tau\rrbracket=\Hb_\tau^\Gamma, \quad~\x\in \Gamma,  \\[2mm]
 \Eb_\tau=&\Eb_\tau^0, \quad~ \x\in \de\Om,
 \label{Mca}
\end{align}
and the corresponding extended problem
 \begin{alignat}{3}
\label{M1ae}
 \curl \Eb+\mu \nabla\alpha &= \I \om \mu \Hb+\K,~~ &\quad \curl \Hb+\ep\nabla\beta &= - \I \om \ep \Eb+\bm J,~~\quad &\x\in\Om\setminus\Gamma,\\[2mm]\label{L2ae}   -\div(\ep\Eb)&=j, &\quad -\div(\mu\Hb)=& k,~\quad~ &\x\in\Om\setminus\Gamma,
\end{alignat}
with nonhomogeneous boundary conditions \rf{La}-\rf{Lc}:   
\begin{alignat}{4}
\label{Lae}
 \llbracket\Eb_\tau\rrbracket&=\Eb_\tau^\Gamma, \quad ~ &\llbracket\n\cdot \ep \Eb\rrbracket&=D_\nu^\Gamma, \quad &~\llbracket\beta\rrbracket&= \beta^\Gamma , &\quad~&\x\in \Gamma, \\[2mm]
 \label{Lbe}
\llbracket\Hb_\tau\rrbracket&=\Hb_\tau^\Gamma, \quad ~ &\llbracket\n\cdot \mu \Hb\rrbracket&=B_\nu^\Gamma, \quad & ~\llbracket\alpha\rrbracket &= \alpha^\Gamma , &\quad~&\x\in \Gamma, \\[2mm]
 \Eb_\tau&=\Eb_\tau^0, \quad &~\n\cdot \mu \Hb&=B_\nu^0, \quad~  &\beta&=\beta^0, &\quad~ &\x\in \de\Om.
 \label{Lce}
\end{alignat}

We assume that the components of the solutions $\Eb, \alpha, \Hb,\beta$ 
belong to the vector or scalar space $ H^1(\Om\setminus\Gamma)$, the boundary functions in \rf{Maa},\rf{Mca}, \rf{Lae}-\rf{Lce} belong to the Sobolev space $H^{1/2}$ on $\Gamma$ or $\de\Om$, and 
 \begin{align}
\label{l2g}
 \K,\bm J,k,j\in L_2(\Om\setminus\Gamma).
\end{align} 
We do not use the space $L_2(\Om)$ in \rf{l2g} to emphasize that the operators on the left-hand side of equations
\rf{M1a}, \rf{M1ae}, \rf{L2ae} are applied in the domain $\Om\setminus\Gamma$ and the jumps of the solutions across $\Gamma$ do not affect the functions $\K,\bm J,k,j$. See also Remark \ref{R1} regarding the number of boundary conditions in the problem \rf{M1ae}-\rf{Lce}.

The following statement contains nothing new compared to Theorem \ref{t1}. We repeat this statement (in a slightly different context) because it is brief and important:
\begin{theorem} \label{t1e}
The boundary value problem \rf{M1ae}-\rf{Lce} is elliptic.
\end{theorem}
\begin{remark} In the case of the elliptic problem  \rf{M1ae}-\rf{Lce}, the choice of spaces for the solutions, for their values on the boundary, and for the inhomogeneities in the equations is entirely natural and standard for elliptic problems. These spaces must be chosen differently for the nonhomogeneous Maxwell problem \rf{M1a}-\rf{Mca}. 
\end{remark}
We still look for solutions to Maxwell's problem in the space $ H^1(\Om\setminus\Gamma)$ and assume that the functions in \rf{Maa},\rf{Mca} belong to $H^{1/2}$. However, we must impose additional constraints on  $\K,~\bm J \in L_2(\Om)$ in \rf{M1a} since the inclusion of $\Eb, \Hb \in H^1$ implies additional smoothness of $\K$ and $\bm J$.
\begin{lemma}\label{l22}
\leavevmode
\begin{itemize}[label={(\roman*)}]
\item[(i)]
If \rf{M1a} holds with $\K,~\bm J\in L_2(\Om)$ and $ ~ \Eb ,~ \Hb\in H^1(\Om\setminus\Gamma)$, then
the functions $ \div \K,~ \div \bm J$ belong to the space $L_2(\Om\setminus\Gamma)$ and
\begin{align}
\label{div1}
\|\div \K\|_{L_2(\Om\setminus\Gamma)}+ \|\div \bm J\|_{L_2(\Om\setminus\Gamma)}\leq C\left(\|\Eb\|_{H^1(\Om\setminus\Gamma)}+\|\Hb\|_{H^1(\Om\setminus\Gamma)}\right),
\end{align}
\item[(ii)] The traces $ K^0_\nu:= K_\nu|_{\de\Om}, J^0_\nu:= J_\nu|_{\de\Om}$ of the scalar values of the normal components of  $\K$ and $\bm J$ on $\de\Om$
and the jumps $ K^\Gamma_\nu:=\llbracket K_\nu\rrbracket, ~  J^\Gamma_\nu:=\llbracket J_\nu\rrbracket$ of the corresponding values on $\Gamma$ 
are well-defined as elements of the Sobolev spaces $H^{-1/2}$ on $\de\Om$ and $\Gamma$, respectively. The norms of all these traces and jumps can be estimated by $\|\Eb\|_{H^1(\Om\setminus\Gamma)}+\|\Hb\|_{H^1(\Om\setminus\Gamma)}$.
\end{itemize}
\end{lemma}
\begin{remark}
The notation $L_2(\Om\setminus\Gamma)$ is used in \rf{div1} instead of $L_2(\Om)$ for the same reason as in \rf{l2g}: to emphasize that the operator $\div$ is applied to $\K, \bm J$ in the domain $\Om\setminus\Gamma$, rather than in $\Om$. In particular, jumps of $\K,\bm J$ on $\Gamma$ do not affect the values of $\div \K$ and  $\div \bm J$. 
\end{remark}
\begin{proof}
We apply the operator $\div$ to both relations \rf{M1a}. Their left-hand sides vanish since $\div\curl=0$. In order to prove \rf{div1}, it only remains to show  that 
\begin{align}
\label{div1a}
\ep\Eb ,~ \mu\Hb\in H^1(\Om\setminus\Gamma).
\end{align}
The inclusions above follow from the condition $\Eb, \Hb \in H^1$ and the assumption  $\ep,~\mu\in W^{1,\infty}(\Om\setminus\Gamma)$ since
$W^{1,\infty}$ is the multiplier algebra of the Sobolev space $H^1$ in a bounded domain.

The second statement of Lemma \ref{l22} follows immediately from \rf{M1a}, \rf{div1a} and Lemma \ref{L00}.
\end{proof}
 
\begin{theorem} \label{t5}
\leavevmode
\begin{itemize}[label={(\roman*)}]
\item[(i)]
Let $(\Eb,\Hb)$ be a solution of the transmission problem \rf{M1a}-\rf{Mca} for Maxwell's equations with $\Eb,\Hb\in H^1(\Om\setminus\Gamma)$. Then $\K, \bm {J}$ have properties described in Lemma \ref{l22}, and the following relations hold for the boundary values of the solutions
\begin{align}
\label{nhom}
K_\nu^\Gamma+\I \om B_\nu^\Gamma=-\div_\tau\Eb_\tau^\Gamma,~\quad 
J_\nu^\Gamma+ \I \om D_\nu^\Gamma=\div_\tau\Hb_\tau^\Gamma,~\quad K_\nu^0+\I \om B_\nu^0=-\div_\tau\Eb_\tau^0,
\end {align}
and the vector $(\Eb,\alpha,\Hb, \beta)$ with $\alpha=const.,~\beta=0$ is a solution of the elliptic problem \rf{M1ae}-\rf{Lce} for which \rf{nhom} holds and \begin{align}
\label{nhom1}
\I\om k= \div\K, \quad \I\om j=-\div\bm J,\quad  ~~ \beta^\Gamma=\alpha^\Gamma=\beta^0=0.
\end{align}
\item[(ii)] 
Conversely, if $(\Eb, \alpha, \Hb, \beta)$ is  a solution of \rf{M1ae}-\rf{Lce} with all components belonging to the vector or scalar space $H^1(\Om\setminus\Gamma) $, functions $\K, \bm J$ satisfy the properties stated in Lemma \ref{l22} and relations \rf{nhom}, \rf{nhom1} hold, then  $(\Eb,\Hb)$ is a solution of the Maxwell problem \rf{M1a}-\rf{Mca}.
\end{itemize}
\end{theorem}
\begin{remark}
    Let us stress that the existence of additional smoothness of ${\bm J}, \K$ in the statement $(i)$ is the consequence of 
    Maxwell's equations. In statement $(ii)$, this additional smoothness, in particular, the existence of the traces of the normal components of $\K, {\bm J}$, is an assumption.
\end{remark}
\begin{proof} Let $(\Eb,\Hb)$ be a solution of the transmission problem \rf{M1a}-\rf{Mca} and $\alpha=const.,~\beta=0$. Then \rf{M1ae} holds. We apply the operator $\div$ to the relations \rf{M1a} and obtain \rf{L2ae} with $k,j$ satisfying \rf{nhom1}. The relations \rf{nhom} follow from \rf{M1a} and Lemmas \ref{l22}, \ref{L00}. The first statement of the theorem is proved. 

To prove the second statement, we apply
 $\div$ to \rf{M1ae} and obtain that
\[
\div (\mu \nabla \alpha) = \div(\I \om \mu \Hb)+\div \K,~~ \quad~
  \div (\ep \nabla \beta)  = - \div(\I \om \ep \Eb)+\div \bm J,~~\quad~ \x\in\Om\setminus\Gamma.
\]
This, \rf{L2ae} and the first two relations in \rf{nhom1} imply
\begin{align}
\label{LapA}
\div (\mu \nabla \alpha) = 
  \div (\ep \nabla \beta) = 0,~~\quad~ \x\in\Om\setminus\Gamma.
\end{align}
  
  Furthermore, from Lemma \ref{L00} and Remark \ref{re2} it follows that  $\llbracket(\curl\Eb)_\nu\rrbracket=-\llbracket\div_\tau\Eb_\tau\rrbracket, ~\x\in\Gamma$. Hence, the equality of the normal components in the first relation in \rf{M1ae} combined with the first relation in \rf{nhom}, implies that $\llbracket(\mu \nabla\alpha)_\nu\rrbracket=0$, i.e. $\llbracket\de\alpha/\de\hat{\n}\rrbracket=0,~\x\in\Gamma$, where $\hat{\n}$ is conormal to $\Gamma$ for the operator $\div (\mu \nabla \cdot)$.
 The same arguments using $\de\Om$ instead of $\Gamma$ lead to $\de\alpha/\de\hat{\n}=0, ~\x\in \de\Om$. If we interchange $\Eb$ and $\Hb$ in the reasoning presented above, we will obtain $\llbracket\de\beta/\de\hat{\n}\rrbracket=0, ~\x\in\Gamma$. Here, $\hat{\n}$ is conormal to $\Gamma$ for the operator $\div (\ep \nabla \cdot)$. Together with the boundary conditions for $\alpha, \beta$ in \rf{nhom1} and relations \rf{LapA}, this leads to the conclusion that $\alpha=const., ~\beta=0$.  Thus,  $(\Eb,\Hb)$ satisfies relations \rf{M1a}-\rf{Mca}.
 \end{proof}

\section{The Resolvent and Eigenvalues}
\label{sec5}
\setcounter{equation}{0}

In this section, we study the spectrum and resolvent of the Maxwell equations and of the corresponding elliptic problem, and therefore, we consider nonhomogeneous equations with homogeneous boundary conditions.

We begin with the elliptic problem and obtain the results for the Maxwell equations as a consequence. It is expedient to rewrite the problem \rf{M1ae}-\rf{Lce}
in matrix form. To this end, we introduce matrix operators $\L_\tau$ and $\L_\nu$ acting on four-dimensional vector functions $\v = (\V, \gamma),~\V \in \C^3, ~\gamma \in \C$, using the matrices $\L_{\ep,\mu}$ and $\L_{\mu,\ep}$ defined in \rf{Lem}. Namely,
\begin{align}
\label{5.1}
    \L_\tau \v = \L_{\mu,\ep} \v
\end{align}
with the domain $D(\L_\tau)$ consisting of functions $\v\in H^1(\Om \setminus \Gamma)$ such that 
\begin{align}
\label{5.2}
    \llbracket\V_\tau\rrbracket=\z, \quad \llbracket\n\cdot \ep \V\rrbracket=0, \quad \llbracket\gamma\rrbracket= 0 , \quad \x\in \Gamma;\quad \qquad \left. \V_\tau \right|_{\de \Om} = \z.
\end{align}
The operator $ \L_\nu$ is defined as follows: 
\begin{align}
    \L_\nu \w = \L_{\ep,\mu} \w
\end{align}
with the domain $D(\L_\nu)$ consisting of functions $\w = (\W, \delta) \in H^1(\Om \setminus \Gamma),~\W \in \C^3, ~\delta \in \C$, such that 
\begin{align}
\label{5.4}
    \llbracket\W_\tau\rrbracket=\z, \quad \llbracket\n\cdot \mu \W\rrbracket=0, \quad \llbracket\delta\rrbracket= 0 , \quad \x\in \Gamma; \quad \qquad  \left. \n\cdot \mu \W\right|_{\de \Om} =\left. \delta \right|_{\de \Om} = 0.
\end{align}
Now, the problem \rf{M1ae}-\rf{Lce} with $\u = (\Eb, \alpha, \Hb, \beta) \in H^1(\Om \setminus \Gamma)$, right-hand side of \rf{M1ae}-\rf{L2ae} in $L^2$, and with homogeneous boundary conditions in \rf{Lae}-\rf{Lce} can be written as
\begin{align}
\label{sys}
    \L \u - \om\, \B \u = \Fb, 
\end{align}
where 
\begin{align}
\label{5.6}
\L = \left(
    \begin{array}{cc}
      \z & \I\, \L_\nu \\[2mm]
      -\I\, \L_\tau & \z
    \end{array}
    \right), 
\end{align}
the operator $\B$ is defined in \rf{Lem}, and
components of $\Fb := \I\, ( {\bm J}, j, -\K, - k )$ 
are defined by the right-hand sides of \rf{M1ae}-\rf{L2ae} and  taken in the prescribed order.

We denote $\L_{\tau^\ast}$ the operator  $\L_\tau$ 
with the matrices, $\ep$ and $\mu$ in its definition  (the domain \rf{5.2} and the map \rf{5.1}) are replaced by their adjoints $\ep^\ast$ and $\mu^\ast$. The operator $\L_{\nu^\ast}$ is defined similarly.
\begin{lemma}
\label{l5.1}
    The operators $\L_\nu$ and $\L_{\tau^\ast}$ are adjoint to each other. The same property holds for $\L_{\tau}$  and $\L_{\nu^\ast}$.
\end{lemma}
\begin{proof}
Denote by $(\cdot, \cdot)$ the scalar product in $L^2 (\Om)$ and 
$\langle \cdot, \cdot \rangle$ the scalar product in $L^2 (\de \Om)$. 
The first statement of the Lemma follows immediately from the following Green's identities that can be written in $\Om$ instead of getting them as the sum of the corresponding identities in $\Om_{-}$ and $\Om_{+} = \Om \setminus \bar{\Om}_-$ due to zero jumps in the first three relations in \rf{5.2} and \rf{5.4}.
\begin{align}
    (\W, \curl \V) - (\curl \W, \V) &= \langle \W, \n \times \V  \rangle =0,
    \quad \text{since} \quad \left. \V_\tau\right|_{\de \Om} = \z, \\[2mm]
     (\ep \nabla \delta, \V) + (\delta, \div (\ep^\ast \V) ) &= \langle \delta ,\n \cdot \ep^\ast \V,  \rangle =0,
    \quad \text{since} \quad \left. \delta \right|_{\de \Om} = 0, \\[2mm]
    (\div (\mu \W), \gamma) + (\W, \mu^\ast\nabla \gamma) &= \langle \n \cdot \mu\W, \gamma \rangle =0,
    \quad \text{since} \quad \left. \n \cdot \mu \W \right|_{\de \Om} = 0.
\end{align}
The proof of the second statement is absolutely similar.
\end{proof}
To solve \rf{sys}, we need to find the kernel and cokernel of the operator $\L$.
\begin{lemma}
    \label{l5.2}
    The kernel and cokernel of the operator $\L$ in \rf{5.6} are 
    one-dimensional and consist of vector functions proportional to ${\bm \ell} = (\z,1,\z,0)$.
\end{lemma}
\begin{proof}
Let us find the kernel of $\L$.
    Equation $\L \u = \z$ with $\u = (\v,\w)$ is equivalent to $\L_\tau \v = \z$  and $\L_\nu \w = \z$. The arguments used to prove Theorem \ref{t0} implies that $\gamma = const.$ and $\delta = 0$. Therefore,
     $\curl \V = \z$ and $\div (\ep \V) = 0$ 
     in $\Om \setminus \Gamma$. From the former relation it follows that $\V = \nabla \phi$, and the 
     second relation leads to $\div (\ep \nabla\phi) = 0$ on $\Om \setminus \Gamma$. 
     Then the equation $\div (\ep \nabla\phi) = 0$ holds throughout the entire $\Om$ due to the first two relations in \rf{5.2}.
     The last condition in \rf{5.2} implies that $\phi$ is constant on $\de \Om$.
     Thus, $\phi$ is constant on $\Om$ and therefore $\V = \z$. Similarly, $\W = \nabla \psi$, with the relation $\div (\mu \nabla \psi) = 0$ valid on $\Om$. Since the conormal derivative of $\psi$ vanishes on $\de \Om$ due to the last relation in \rf{5.4}, we obtain that $\psi = const.$ and, therefore, $\W = \z$. The statement concerning the kernel of $\L$ is proved. To find the cokernel of $\L$, one can look for the kernel of the adjoint operator, which has the form of \rf{5.6} with $\nu,\tau$ replaced by $\nu^\ast, \tau^\ast$, respectively. Thus, the arguments above imply that the cokernel is also proportional to ${\bm \ell}$.
\end{proof}
Denote by $L_{2,\perp} (\Om), H^{1}_{\perp} (\Om\setminus \Gamma)$ the 
spaces of functions in $L_2, H^1$, respectively, that are $L_2$-orthogonal to ${\bm \ell}$. Clearly, functions proportional to  ${\bm \ell} = (\z,1,\z,0)$ belong to the kernel and cokernel of the operator $\L- \om\, \B$ for any $\om$. We consider $\om$ in \rf{sys} as a spectral parameter. It is natural to consider the resolvent $\Res_\om$ in the spaces of functions orthogonal to ${\bm \ell}$: 
\begin{align}
\label{5.10}
         \u = \Res_\om\, {\bm F}, \quad \Res_\om := (\L - \om \B)^{-1}: ~  L_{2,\perp} \to H^{1}_{\perp}.
\end{align}
 The eigenfunctions of the operator $\L  - \om  \B $ will be defined as nontrivial solutions of the homogeneous equation \rf{sys} that are orthogonal to ${\bm \ell}$. Otherwise, any $\om$ is an eigenvalue with an eigenfunction proportional to ${\bm \ell}$.
\begin{theorem} 
\label{t5.1}
The resolvent \rf{5.10} is
a meromorphic function of $\om$. If $\ep$ and $\mu$ are real positive-definite matrices, then $\Res_\om$ is finite-meromorphic in $\om$ with real simple poles located symmetrically on the real line and with residues of finite ranks that are complex-conjugate for the symmetric $\om$.
\end{theorem}
\begin{proof}
To prove the meromorphic dependence of $\Res_\om$ we rewrite \rf{5.10} as
\begin{align}
\label{u}
    \Res_\om = \L^{-1} (I - \om \,\B\, \L^{-1})^{-1}: ~  L_{2,\perp} \to H^{1}_{\perp} .
\end{align}
This formula can also be obtained by looking for the solution of \rf{sys} in the form $\u = \L^{-1} \ut, ~ \ut \in  L_{2,\perp}$.
The operator $\B  \L^{-1} $ is compact in $L_{2,\perp}$ since $\L^{-1}$ is compact there (due to the compactness of the embedding of the space $H^1$ into $L_2$) and $\B$ is bounded. 
Thus, the analytic Fredholm theorem implies that $\Res_\om$ is meromorphic in $\om$.

Now, let $\ep$ and $\mu$ be positive-definite matrices. 
Then $\B^{1/2}$ is well-defined by replacing $\ep$ and $\mu$ with $\ep^{1/2}$ and $\mu^{1/2}$ in \rf{Lem}. 
From \rf{sys} it follows that 
\begin{align}
% \label{sys}
    \B^{1/2} \u
    - \om\, \B^{1/2} \L^{-1}\B^{1/2} \B^{1/2} \u =  \B^{1/2}\L^{-1}\, {\bm F},
\end{align}
and therefore, 
\begin{align}
\label{5.14}
    \B^{1/2} \u = \left( I - \om\, \Q \right)^{-1} \B^{1/2}  \L^{-1}{\bm F}.
\end{align}
Here, the operator $\Q  := \B^{1/2} \,\L^{-1}\B^{1/2}$ is self-adjoint in $L_{2,\perp}$ since the operator $\L^{-1}$ is self-adjoint due to \rf{5.6} and Lemma \ref{5.1}. In addition, $\Q$ is compact. Thus, the operator $(I - \om \Q)^{-1}$ in $L_{2,\perp}$ is finite-meromorphic in $\om$ and has simple real poles with residues of finite ranks. Since \rf{sys},\rf{5.14} imply that
\begin{align}
% \label{5.12}
    \u = \L^{-1} (\Fb + \om\,\B \u) = 
    \L^{-1} (I + \om\,\B^{1/2} \left( I - \om\, \Q \right)^{-1} \B^{1/2}  \L^{-1}) {\bm F},
\end{align}
where $ (I - \om \Q)^{-1}$ is a finite-meromorphic operator in $L_{2,\perp}$ and $\L^{-1}, \B^{1/2}:  L_{2,\perp} \to H^{1}_{\perp} $ are bounded. This justifies 
all the properties of $\Res_\om$ except the symmetry of the eigenvalues $\om$ with respect to the origin and the complex conjugation of the corresponding eigenfunctions. The latter properties 
follow from \rf{sys} and \rf{5.6} since 
these relations show that the equation for the eigenvalue $-\om$ is equivalent to the conjugate equation for the eigenvalue $\om$.
\end{proof}

Denote by $L_{2,\B}$ the subspace
of eight-dimensional vector functions
$\u = (\Eb, \alpha, \Hb, \beta)  \in L_2$ for which $\alpha = \beta = 0$ with the weighted dot product defined by
\begin{align}
\label{5.15}
    \left.(\u_1, \u_2)\right|_{L_{2,\pperp}} = \left.(\B \u_1, \u_2)\right|_{L_{2}}.
\end{align}
\begin{theorem}
    \label{t5.2}
    Let $\ep$ and $\mu$ be real positive-definite matrices. 
    Then the eigenfunctions of \rf{sys} belong to $L_{2,\B}$ and one can construct an orthogonal basis in $L_{2,B}$  from the eigenfunctions of the elliptic system \rf{sys}.
\end{theorem}
\begin{proof}
    Let $\u = (\Eb, \alpha, \Hb, \beta) \in H^1 (\Om \setminus \Gamma)$ be a nontrivial solution of $\L \u - \om \B \u = \z$, i.e.
\begin{align}
    \curl \Eb + \mu \nabla \alpha -\I \om \mu \Hb &= \om \ep \Eb, \\[2mm]
    -\div (\ep \Eb) &= 0, \\[2mm]
   \curl \Hb + \ep \nabla \beta +\I \om \ep \Eb &= \om \mu \Hb, \\[2mm]
   -\div (\ep \Hb) &= 0.
\end{align}
We apply the operator $\div$ to the first and third lines and obtain $\div (\mu \nabla \alpha ) =\div (\ep \nabla \beta ) =0.$ After that, one can repeat the arguments used in the last paragraph of the proof of Theorem \ref{t0} to show that $\alpha = const.$ and $\beta = 0$. By definition, $\u$ is orthogonal to ${\bm \ell}$, and therefore $\u \in L_{2,\B}$. 

Let us prove now that there exists an orthogonal basis in $L_{2,\B}$ consisting of eigenfunctions of \rf{sys}. 
Consider the self-adjoint compact operator $\Q$ in $L_{2,\perp}$ defined in \rf{5.14} and equation $\Q \bpsi_i = \lambda_i \bpsi_i$ for its eigenfunctions.
 Since $L_{2,\B}$ and its orthogonal complement $L_{2,\perp}$ are invariant subspaces for $\Q$, 
one can form an orthogonal in $L_2$ basis for the space $L_{2,\B}$ consisting of eigenfunctions $\bpsi_i$. The orthogonal complement to $L_{2,\B}$ consists of the eigenfunctions of $\Q$ with a zero eigenvalue. If $\bpsi_i \in L_{2,\B}$ then $\lambda_i \neq 0$, since $\Q$ is invertible on $L_{2,\B}$. Thus, $\bpsi_i - \om_i\, B^{1/2} \L^{-1} B^{1/2} \psi_i = \z$, where $\om_i = 1/\lambda_i.$ Then $\u_i = B^{-1/2} \bpsi_i$  are the eigenfunctions of \rf{sys} such that $B^{1/2} \u_i$ are orthogonal in $L_2$, i.e., $\{ \u_i\}$ is the basis we need.
\end{proof}
Below, we use the same notation $L_{2,B}$ from \rf{5.15} for the weighted space of six-dimensional solutions $\u = (\Eb, \Hb) \in L_2$ of Maxwell's equations. Combining theorems \ref{t5}, \ref{t5.1}, \ref{t5.2} we obtain the following:
\begin{theorem}
    The spectrum of the Maxwell system is discrete for arbitrary complex-valued $\ep$ and $\mu$. 
    
    If $\ep$ and $\mu$ are real positive-definite matrices, then the eigenvalues are real and symmetric with respect to the origin, with complex-conjugated eigenfunctions for symmetric $\om$.
    One can construct an orthogonal basis in $L_{2,B}$ from the eigenfunctions of the Maxwell system.
\end{theorem}

% \bibliography{asympexpan}
% \bibliographystyle{elsarticle-num} 

 \newcommand{\noop}[1]{}

\end{document}